\documentclass[12pt]{article}

\usepackage{epsfig,amsthm,amsmath,amssymb,graphics,verbatim,amsfonts,subfigure,psfrag}
\usepackage[letterpaper,margin=1in]{geometry}
\usepackage[all]{xy}

\newtheorem{thm}{Theorem}[section]

\newtheorem{prop}[thm]{Proposition}
\newtheorem{lem}[thm]{Lemma}

\def\R{\mathbb{R}}

\def\N{\mathbb{N}}

\def\H{\mathbb{H}}
\def\I{\infty}

\def\txtT{{\textnormal{T}}}
\def\txtD{{\textnormal{D}}}

\newcommand{\be}{\begin{equation}}
\newcommand{\ee}{\end{equation}}
\newcommand{\bea}{\begin{eqnarray}}
\newcommand{\eea}{\end{eqnarray}}
\newcommand{\beann}{\begin{eqnarray*}}
\newcommand{\eeann}{\end{eqnarray*}}
\newcommand{\benn}{\begin{equation*}}
\newcommand{\eenn}{\end{equation*}}

\def\ra{\rightarrow}
\def\I{\infty}

\newcommand{\cB}{{\mathcal B}}  
\newcommand{\cH}{{\mathcal H}}  
\newcommand{\cI}{{\mathcal I}}  
\newcommand{\cS}{{\mathcal S}}  

\begin{document}

\author{Christian Kuehn\thanks{Institute for Analysis and Scientific Computing, 
Vienna University of Technology, 1040 Vienna, Austria}}
 
\title{A Remark on Geometric Desingularization of a Non-Hyperbolic Point
using Hyperbolic Space}

\maketitle

\begin{abstract}
A steady state (or equilibrium point) of a dynamical system is hyperbolic if the Jacobian
at the steady state has no eigenvalues with zero real parts. In this case,
the linearized system does qualitatively capture the dynamics in a small neighborhood of 
the hyperbolic steady state. However, one is often forced to consider 
non-hyperbolic steady states, for example in the context of bifurcation theory. A 
geometric technique to desingularize non-hyperbolic points is the blow-up method. The classical
case of the method is motivated by desingularization techniques arising in algebraic geometry. 
The idea is to blow up the steady state to a sphere or a cylinder. In the blown-up space, one
is then often able to gain additional hyperbolicity at steady states. In this paper, we discuss an 
explicit example where we replace the sphere in the blow-up by hyperbolic space. It is shown 
that the calculations work in the hyperbolic space case as for the spherical case. This approach may be 
even slightly more convenient if one wants to work with directional charts. Hence, it
is demonstrated that the sphere should be viewed as an auxiliary object in the blow-up 
construction. Other smooth manifolds are also natural candidates to be inserted at steady states.
\end{abstract}

\section{Introduction}

Consider an ordinary differential equation (ODE) given by
\be
\label{eq:ODE}
\frac{dz}{dt}=z'=f(z),
\ee
where $z=z(t)\in\R^N$, $N\in\N$, $t\in\R$ and $f:\R^N\ra \R^N$ is assumed to be sufficiently smooth. 
Suppose $z^*\in\R^N$ is a steady state (or equilibrium point) of \eqref{eq:ODE}, {i.e.}, 
$f(z^*)=0$. Using a translation of coordinates, if necessary, we may assume for the following 
analysis without loss of generality that $z^*=0:=(0,0,\ldots,0)\in\R^N$. The first standard 
calculation for steady states is to consider the linearized system in a neighborhood of the 
steady state 
\be
\label{eq:ODE_linear}
Z'=(\txtD f_{0})Z,
\ee  
where $Z\in\R^N$ and $\txtD f_0\in\R^{N\times N}$ denotes the total derivative of $f$ evaluated 
at $z=0$. It is also common to refer to $\txtD f_0$ as the Jacobian matrix or simply the Jacobian.
Let $\lambda_n$ for $n\in\{1,2,\ldots,N\}$ denote the eigenvalues of $\txtD f_0$. If they eigenvalues
have no zero real parts, $\textnormal{Re}(\lambda_n)\neq 0$ for all $n$, then the steady state
$z^*=0$ is called hyperbolic. The Hartman-Grobman Theorem (see {e.g.}~\cite[p.120-121]{Perko}) implies
that in a neighborhood of a hyperbolic steady state, the flows generated by \eqref{eq:ODE} and
\eqref{eq:ODE_linear} are topologically conjugate. For most practical purposes this implies that
 we may just the linear ODE \eqref{eq:ODE_linear} to study the dynamics near $z^*=0$.\medskip

However, non-hyperbolic points are unavoidable if we want to analyze bifurcation points 
\cite{GH,Kuznetsov}. The linearization approach breaks down and one has to carefully consider
the influence of nonlinear terms. One possible technique that can be very successful in this 
context is geometric desingularization; see {e.g.}~\cite[p.67-70]{Dumortier1} for a particular
example or \cite{Dumortier} for general planar singularities. We are going to introduce 
geometric desingularization via the blow-up method in more detail in Section \ref{sec:spherical}.\medskip

The main geometric idea of the method arose in algebraic geometry in the context of 
desingularization of algebraic varieties \cite[p.29]{Hartshorne}, where one replaces 
certain singular points by projective space. The resulting variety either has no singular
points anymore or one can try to repeat the blow-up. Under certain conditions one may indeed
reach a complete desingularization as stated in the celebrated Hironaka Theorem 
\cite{Hironaka1,Hironaka2}.\medskip 

In the context of ODEs, the classical strategy involves using a spherical blow-up
as one works in real space and not in the context of (complex) projective space. The key 
difference to the algebraic geometry blow-up is that one also has to keep track 
of the dynamics on the blown-up space. There has been a tremendous amount of work
on using the blow-up technique for planar ODEs \cite{Dumortier,Dumortier1,Dumortier5}, 
canard solutions \cite{DumortierRoussarie,KruSzm3,KrupaWechselberger,Wechselberger2}, 
traveling wave problems \cite{DeMaesschalckPopovicKaper,DumortierPopovicKaper} and a 
large variety of other problems in the theory of multiple time scale dynamical systems 
\cite{DeMaesschalckDumortier3,GucwaSzmolyan,KruSzm4,KuehnUM}.\medskip   

Using spherical, or cylindrical, spaces are currently the standard choices
to desingularize non-hyperbolic steady states of ODEs. However, there seems to be 
now apparent reason why other manifolds could function equally well, or even better.
In this paper, we investigate this idea in more detail and consider a simple example
to illustrate the main idea. The spherical case is discussed in Section 
\ref{sec:spherical}, which is also a fully self-contained introduction to the 
blow-up method. In Section \ref{sec:hyperbolic} we replace the sphere by 
hyperbolic space, {i.e.}~by using a manifold with constant negative curvature.
We emphasize that the word `hyperbolic' is then used in two distinct ways: (1)
for the dynamical type of a steady state and (2) for a smooth manifold which replaces
the sphere in the blown-up space. The results in Section \ref{sec:hyperbolic} confirm
the intuition that using a spherical blown-up space is not crucial and hyperbolic
space works also for geometric desingularization in the example. This indicates that
one should be open-minded about trying to use different manifolds for geometric 
desingularization.\medskip

\textit{Acknowledgments:} I would like to thank the Austrian Academy of Sciences ({\"{O}AW}) 
for support via an APART fellowship. I also acknowledge the European Commission (EC/REA) for 
support by a Marie-Curie International Re-integration Grant.   

\section{Spherical Blow-Up}
\label{sec:spherical}

In this section a basic test example for the blow-up method is 
reviewed from \cite{Dumortier1} and more explicit calculations for this 
example are provided. The spherical blow-up is constructed in this context, which leads 
to a geometric desingularization of the problem.\medskip

Consider the following planar ODE \cite{Dumortier1} for $z(t)=(x(t),y(t))\in\R^2$
\be
\label{eq:main}
\begin{array}{lclclcr}
\frac{dx}{dt}&=&x'&=& ax^2 -2xy&=:&f_1(x,y),\\
\frac{dy}{dt}&=&y'&=& y^2-axy&=:&f_2(x,y),\\
\end{array}
\ee
where $a>0$ is a positive parameter, we abbreviate $(x,y)=(x(t),y(t))$ and we denote the vector 
field by $f:=(f_1,f_2)^T$, where $(\cdot)^T$ denotes the transpose. 
We may view the vector field $f$ as a smooth section into the tangent bundle $f:\R^2\ra \txtT \R^2$.
If $p\in\R^2$ is a given point, then we shall usually employ the natural identification of 
the tangent space $\txtT_p\R^2\cong \R^2$.\medskip
  
Observe that $(x,y)=(0,0):=0$ is a steady state, {i.e.}~$f_1(0)=0=f_2(0)$), for \eqref{eq:main}. 
It is straightforward to compute the linearized system $Z=(X,Y)\in\R^2$ at the origin 
\benn
\left(\begin{array}{c} X' \\ Y' \\\end{array}\right)=(\txtD f)_{0}
\left(\begin{array}{c} X \\ Y \\\end{array}\right)=
\left(\begin{array}{cc} 2ax-2y & -2x \\ -ay & 2y-ax\\\end{array}\right)_{0}
\left(\begin{array}{c} X \\ Y \\\end{array}\right)=
\left(\begin{array}{cc} 0 & 0 \\ 0 & 0\\\end{array}\right)
\left(\begin{array}{c} X \\ Y \\\end{array}\right),
\eenn
where we shall always employ capital variables $Z=(X,Y)\in\R^2$ to emphasize when we work with
a linearized problem. We see that the origin is a non-hyperbolic steady state since $\txtD F_0$ 
has two zero eigenvalues; see also Figure \ref{fig:fig1}(a). Hence, further analysis is required and the 
blow-up method provides one approach to understand the dynamics.\medskip

\begin{figure}[htbp]
	\centering
\psfrag{a}{(a)}
\psfrag{b}{(b)}
\psfrag{c}{(c)}
\psfrag{phi}{$\Phi$}
\psfrag{f}{$f$}
\psfrag{fhat}{$\hat{f}$}
\psfrag{fbar}{$\bar{f}=\frac1r \hat{f}$}
		\includegraphics[width=0.98\textwidth]{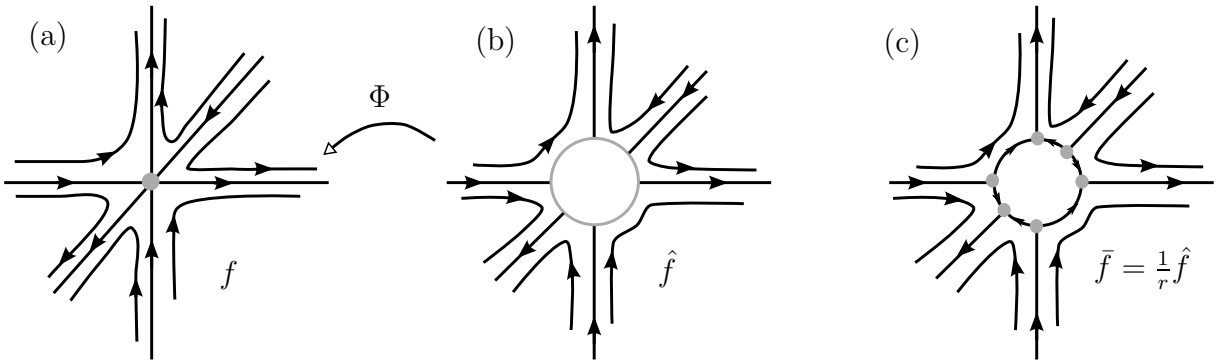}
		\caption{\label{fig:fig1}Sketch of the main steps of the (spherical) blow-up method 
		for the example \eqref{eq:main}. (a) Original vector field $f$ with non-hyperbolic 
		steady state (gray) at the origin. (b) Blown-up vector field $\hat{f}$ on $\cB$ with
		a full circle of steady states (gray) given by $\cS^1\times \{r=0\}$. (c) Desingularized
		blown-up vector field $\bar{f}$ with precisely six hyperbolic saddle steady states 
		(gray). The small arrows on $\cS^1\times \{r=0\}$ indicate the qualitative part of the flow
		which is different from $\hat{f}$. Observe that the flow directions are compatible with
		the phase portrait for $\cS^1\times \{r>0\}$.}
\end{figure}

For planar vector fields, the classical approach of the blow-up method is to use a 
transformation which replaces the point $p$ with a (unit) circle
\benn
\cS^1=\{(x,y)\in\R^2:x^2+y^2=1\}=\{(x,y)\in\R^2:x=\cos\theta,y=\sin\theta,\theta\in[0,2\pi)\}.
\eenn
In higher-dimensional cases, one usually uses spheres or cylinders. Formally, we 
fix $r_0>0$, consider the interval $\cI:=[0,r_0]$ and define the manifold
\be
\cB:=\cS^1\times \cI.
\ee 
Sometimes other choices for $\cI$ are convenient such as $\cI=\R$, $\cI=[-r_0,r_0]$ 
or $\cI=[0,\I)$ but in our context $\cI:=[0,r_0]$ will suffice. A spherical blow-up 
transformation is given by
\benn
\Phi:\cB\ra \R^2,
\eenn
where the map $\Phi$ will be defined algebraically below. We already note that
if $\Phi$ is differentiable then the push-forward $\Phi_*:\txtT\cB\ra \R^2$ induces
a vector field $\hat{f}$ on the blown-up space $\cB$ if we require the condition
\benn
\Phi_*\left(\hat{f}\right)=f.
\eenn 
One possibility is to define $\Phi$ algebraically is to use the weighted polar blow-up. 
Let $(\theta,r)\in\cS^1\times [0,r_0]$ be coordinates for $\cB$ and define
\benn
\Phi(\theta,r)=(r^\alpha\cos\theta,r^\beta\sin\theta)=(x,y),
\eenn
where $\alpha,\beta\in\R$ are the weights to be chosen below and $\theta\in[0,2\pi)$. 
Observe that $\Phi$ is a diffeomorphism outside of the circle $\cS^1\times \{r=0\}$, 
which corresponds to the steady state $p=(0,0)$. Hence, the polar blow-up 
transformation indeed inserts a circle at the non-hyperbolic point and topologically
conjugates the dynamics between
\benn
\R^2-\{(0,0)\}\qquad \text{and} \qquad \cB-\left[\cS^1\times \{r=0\}\right].
\eenn
To determine good weights $\alpha$ and $\beta$ one may use quasi-homogeneity of 
the vector field; recall that $f$ is quasi-homogeneous of type $(\alpha,\beta)$
and degree $k+1$ if 
\be
\label{eq:quasi}
f(r^\alpha x,r^\beta y)=(r^{\alpha+k}f_1(x,y),r^{\beta+k}f_2(x,y))^T.
\ee
Substituting the vector field \eqref{eq:main} into \eqref{eq:quasi} yields
\be
\begin{array}{rcl}
r^{2\alpha}ax^2 -r^{\alpha+\beta}2xy&=&r^{\alpha+k}(ax^2-2xy),\\
r^{2\beta}y^2 -r^{\alpha+\beta}axy&=&r^{\beta+k}(y^2-axy).\\
\end{array}
\ee
Therefore, the vector field $f$ is quasi-homogeneous of type $(\alpha,\beta)=(1,1)$
and degree 2 (with $k=1$). Then one chooses the blow-up weights as the type of 
the quasi-homogeneous vector field so that for \eqref{eq:main} we just have
a polar coordinate change
\benn
\Phi(\theta,r)=(r\cos\theta,r\sin\theta)=(x,y).
\eenn

\begin{lem}
\label{lem:polar}
The vector field $\hat{f}$ in polar coordinates is given by
\be
\label{eq:polar_vf}
\begin{array}{lcl}
\theta'&=& r\left(3\cos\theta\sin^2\theta-2a\sin\theta\cos^2\theta\right),\\
r'&=&r^2(a\cos\theta -2\sin\theta-2a\cos\theta\sin^2\theta+3\sin^3\theta).\\
\end{array}
\ee
\end{lem}

\begin{proof}
One possibility is to note that $\hat{f}(\theta,r)=(\txtD\Phi)^{-1}f(\Phi(\theta,r))$ 
and calculate. Alternatively, one may proceed slightly more directly
\be
\begin{array}{rclcl}
ar^2\cos^2\theta-2r\cos\theta\sin\theta&=&x'&=&r'\cos\theta-r\theta'\sin\theta,\\
r^2\sin^2\theta-ar^2\sin\theta\cos\theta&=&y'&=&r'\sin\theta+r\theta'\cos\theta,\\
\end{array}
\ee 
and proceed to solve for $\theta'$ and $r'$.
\end{proof}

The ODE \eqref{eq:polar_vf} has an entire circle of steady states given by
$\cS^1\times \{r=0\}$; see Figure \ref{fig:fig1}(b). However, it is possible to desingularize the 
vector field $\hat{f}$ by division by $1/r$, {i.e.} we define 
\benn
\bar{f}:=\frac{1}{r}\bar{f}. 
\eenn
The division by $1/r$ does not change the qualitative dynamics on the set 
$\cS^1\times \{r>0\}$ up to a time rescaling \cite[Sec.1.4.1]{Chicone2}. However, 
the $1/r$ scaling does drastically change the dynamics on the circle 
$\cS^1\times \{r=0\}$. The desingularized vector field $\bar{f}$ is given by
\be
\label{eq:polar_vf_desing}
\begin{array}{lcl}
\theta'&=& 3\cos\theta\sin^2\theta-2a\sin\theta\cos^2\theta,\\
r'&=&r(a\cos\theta -2\sin\theta-2a\cos\theta\sin^2\theta+3\sin^3\theta).\\
\end{array}
\ee
Having computed \eqref{eq:polar_vf_desing}, the dynamics follows by direct 
calculation of the steady states and linearization.

\begin{prop}
For $a>0$ fixed, There are six steady states for \eqref{eq:polar_vf_desing} on $\cS^1\times \{r=0\}$. 
Four are given by
\benn
\theta=0,\frac{\pi}{2},\pi,\frac{3\pi}{2}
\eenn
while the remaining two are defined by the condition $\tan \theta=\frac23a$. The six 
steady states are hyperbolic saddle points as shown in Figure \ref{fig:fig1}(c). 
\end{prop}

However, although the calculations using polar coordinates are easy for
our example problem, they become quickly very involved for other problems. In particular, 
consider the situation when the blow-up has to be used iteratively when new steady states 
on the sphere associated to $\{r=0\}$ are also non-hyperbolic.\medskip

It is more convenient to use charts for $\cB$ in combination with a so-called weighted 
directional blow-up. Introduce coordinates on $\cB$ given 
by $(\bar{x},\bar{y},\bar{r})\in \cS^1\times [0,r_0]$ with $\bar{x}^2+\bar{y}^2=1$. Then 
define the weighted directional blow-up map by
\be
\label{eq:dbup}
\Psi:\cB\ra \R^2,\qquad \Psi(\bar{x},\bar{y},\bar{r})=(\bar{r}\bar{x},\bar{r}\bar{y}). 
\ee   
So how should we define charts $\kappa_i:\cB\ra \R^2$ to make the calculations as 
simple as possible? One approach is to require that the induced local coordinate changes
\benn
\psi_i=\Psi\circ\kappa_i^{-1}
\eenn 
are easy to compute and the vector fields $\txtD\psi_i^{-1}f\psi$ have a tractable 
algebraic form. Let $x_i,y_i\in\R$, $r_i\in[0,r_0]$ and let $(r_1,y_1)$, $(r_2,x_2)$ be 
coordinates on $\R^2$. One possibility is to design the charts is to consider 
\eqref{eq:dbup} and try to require
\be
\label{eq:psii}
\psi_1(r_1,y_1)=(r_1,r_1y_1)\qquad\text{and}\qquad \psi_2(r_2,x_2)=(r_2x_2,r_2).   
\ee 
The following diagram illustrates the main aspects of the weighted directional blow-up:
\benn
\xymatrix{& & & \cB=\cS^1\times[0,r_0] \ar[llld]^{\kappa_2} \ar[ld]^{\kappa_1} \ar[rd]^{\Psi} & \\
 (r_2,x_2)\in\R^2 \ar[rr]_{\kappa_{21}} & & (r_1,y_1)\in\R^2 \ar[ll]_{\kappa_{12}} 
\ar[rr]^{\psi_1} & & (x,y)\in\R^2,  
}
\eenn
where $\kappa_{12}$ and $\kappa_{21}$ denote the transition maps between the two charts 
$\kappa_1$ and $\kappa_2$. If \eqref{eq:psii} holds then this leads to
\be
\label{eq:def_charts}
\begin{array}{l}
\kappa_1(\bar{x},\bar{y},\bar{r})=\psi_1^{-1}\circ\Psi(\bar{x},\bar{y},\bar{r})=
\psi_1^{-1}(\bar{r}\bar{x},\bar{r}\bar{y})=(\bar{r}\bar{x},\bar{r}\bar{y}/(\bar{r}\bar{x}))=
(\bar{r}\bar{x},\bar{y}/\bar{x}),\\
\kappa_2(\bar{x},\bar{y},\bar{r})=\psi_2^{-1}\circ\Psi(\bar{x},\bar{y},\bar{r})=
\psi_2^{-1}(\bar{r}\bar{x},\bar{r}\bar{y})=(\bar{r}\bar{x}/(\bar{r}\bar{y}),\bar{r}\bar{y})=
(\bar{x}/\bar{y},\bar{r}\bar{y}).\\
\end{array}
\ee
Hence we may use \eqref{eq:def_charts} as definitions of the charts and obtain that the 
corresponding coordinate changes on $\R^2$ are given by \eqref{eq:psii}. 

\begin{lem}
The vector fields using the charts $\kappa_{1,2}$ are given by
\be
\label{eq:undesing}
\left\{\begin{array}{lcl}
r_1' &=& r_1^2(a-2y_1),\\
y_1'&=&r_1y_1(3y_1-2a),\\
\end{array}\right.\qquad
\left\{\begin{array}{lcl}
r_2' &=& r_2^2(1-ax_2),\\
x_2'&=&r_2x_2(2ar_2-3).\\
\end{array}\right.
\ee
\end{lem}
 
\begin{proof}
As before, we may formally carry out the coordinate change. Or one may use 
direct calculations, for example, we have
\benn
r_2'=y'=r_2^2-ar_2^2x_2,\qquad x'=r_2'x_2+r_2x_2'=ar_2^2x_1^2-2r_2^2x_2. 
\eenn 
From these results, the vector field in $(r_2,x_2)$-coordinates easily follows.
The calculation for the $\kappa_1$-chart is similar.
\end{proof}

The ODEs \eqref{eq:undesing} are still polynomial vector fields and algebraically
a lot simpler to treat in comparison to long expressions using trigonometric 
functions. As for the polar case, we may again desingularize the problem using 
a division  by $1/r_i$. For the first chart this yields
\be
\label{eq:desing1}
\begin{array}{lcl}
r_1' &=& r_1(a-2y_1),\\
y_1'&=&y_1(3y_1-2a).\\
\end{array}
\ee
We have that \eqref{eq:desing1} is defined in $(r_1,y_1)\in[0,r_0]\times \R$.
We may consider this domain as corresponding to covering the right-half plane 
of $\cB\subset \R^2$ outside of the open half-disc $\{x>0,x^2+y^2<1\}$; 
see Figure \ref{fig:fig2}.\medskip

\begin{figure}[htbp]
	\centering
\psfrag{a}{(a)}
\psfrag{b}{(b)}
\psfrag{kappa}{$\kappa_1$}
\psfrag{r1}{$r_1$}
\psfrag{y1}{$y_1$}
		\includegraphics[width=0.7\textwidth]{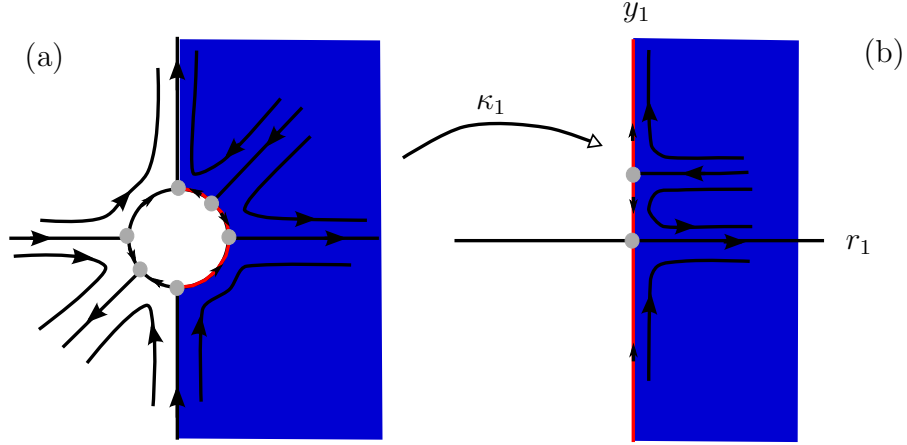}
		\caption{\label{fig:fig2}Sketch of the coordinate chart $\kappa_1$ associated to the 
		$x$-directional blow-up. (a) Blown-up space $\cB$ with phase portrait (black). (b) Directional
		coordinates $(r_1,y_1)\in\R^2$; the blue region corresponds to the blue region in (a) using the
		chart map $\kappa_1$, respectively its inverse $\kappa_1^{-1}$. Note that the half-circle from
		(a) is mapped to the vertical $y_1$-axis.}
\end{figure} 

There are two steady states for \eqref{eq:desing1} given by
\benn
(r_1,y_1)=(0,0),\qquad (r_1,y_1)=\left(0,\frac23a\right)
\eenn
which correspond to the steady states with angles $\theta=0$ and the smallest
positive zero of $\tan\theta=\frac23a$. In the form \eqref{eq:desing1} it is 
easier to check the eigenvalues of the linearized system
\benn
\left(\begin{array}{c} R_1' \\ Y_1' \\\end{array}\right)=
\left(\begin{array}{cc} a-2y_1 & -2r_1 \\ 0 & 6y_1-3a\\\end{array}\right)
\left(\begin{array}{c} R_1 \\ Y_1 \\\end{array}\right)
\eenn
to conclude that the two steady states are hyperbolic saddle points. The calculations
for the second desingularized system
\be
\label{eq:desing2}
\begin{array}{lcl}
r_2' &=& r_2(1-ax_2),\\
x_2'&=&x_2(2ar_2-3),\\
\end{array}
\ee
are similar and we also find two saddle points. The system \eqref{eq:desing1} covers the 
outside of the open half-disc $\{y>0,x^2+y^2<1\}$ similar to the case shown in Figure \ref{fig:fig2} just for the
upper half-plane. We can define two 
more charts, which also cover the left-half plane and the lower half-plane. If we define
\be
\label{eq:def_charts1}
\begin{array}{l}
\kappa_{3}(\bar{x},\bar{y},\bar{r})=(-\bar{r}\bar{x},\bar{y}/\bar{x}),\\
\kappa_{4}(\bar{x},\bar{y},\bar{r})=(\bar{x}/\bar{y},-\bar{r}\bar{y}),\\
\end{array}
\ee
then the local coordinate changes are given by
\be
\label{eq:psii_neg}
\psi_{3}(r_3,y_3)=(-r_3,r_3y_3)\qquad\text{and}\qquad \psi_4(r_4,x_4)=(r_4x_4,-r_4).   
\ee 
With the four charts, one easily checks that there are six hyperbolic saddle 
points on $\cB\times\{r=0\}$ and one determines the direction of the flow as shown
in Figure \ref{fig:fig1}(c).\medskip

As a remaining question we consider the relation between the directional and polar
blow-up maps. For example, if we would like to change from polar coordinates $(\theta,r)$
to Euclidean coordinates $(r_1,y_1)$, we would like the following diagram to commute:
\benn
\xymatrix{& \cB=\cS^1\times[0,r_0] \ar[ld]^{\alpha_1} \ar[rd]^{\Phi} & \\
 (x_1,r_1)\in\R^2 \ar[rr]_{\psi_1} & & (x,y)\in\R^2.}
\eenn
In particular, this yields the requirement
\benn
\Phi(\theta,r)=(r\cos\theta,r\sin\theta)=(x,y)=(r_1,r_1y_1)=\psi_1(r_1,y_1).
\eenn 
Therefore, we must have $r_1=r\cos\theta$ which implies
\benn
r_1y_1=y_1r\cos\theta=r\sin\theta\quad\Rightarrow \quad y_1=\tan\theta.
\eenn
The coordinate change 
\be
\alpha_1(\theta,r)=(r\cos\theta,\tan\theta)=(r_1,y_1)
\ee
is not well-defined when $\theta=\pi/2,3\pi/2$ but it is a diffeomorphism 
otherwise. Note that this implies the polar blow-up is indeed equivalent to
the directional blow-up in the $x$-direction expect on the vertical $y_1$-axis.
This is geometrically clear as we cannot map the circle diffeomorphically, or 
even homeomorphically, onto the $y_1$-axis. In some sense, this fact leads one
to the viewpoint that using a spherical blow-up, if one eventually wants to 
calculate in directional coordinates anyway, is not the only choice for the
blown-up space. In fact, there may be manifolds that work more naturally with
directional coordinate charts. 

\section{Hyperblic Space Blow-Up}
\label{sec:hyperbolic}

In this section we address the question whether it is possible to consider a 
blown-up space other than the sphere to analyze the dynamics. As we shall show below, 
the answer to this question is positive. The second question is whether other 
blow-up spaces are more convenient from a practical and/or theoretical perspective.
Again, this question has at least a `non-negative' answer, {i.e.}~we shall show that
for our test example, the calculation for hyperbolic space work equally well; in fact,
it may be even more convenient to use hyperbolic space if we have distinguished directions 
and want to work in charts.\medskip

Instead of the sphere, we shall now work with hyperbolic space \cite{Ratcliffe} 
via the hyperboloid model and define
\benn
\H_x:=\{(x,y)\in\R^2:x^2-y^2=1\},\qquad \H_y:=\{(x,y)\in\R^2:y^2-x^2=1\}. 
\eenn
Furthermore, we define the associated blow-up spaces
\benn
\cB_x:=\H_x\times [0,\rho_0],\qquad \cB_y:=\H_y\times [0,\rho_0]
\eenn
for some fixed $\rho_0>0$; note that $\rho_0$ plays the same role as $r_0$ for the 
spherical case. We start with the blow-up using just the space $\cB_x$. Note that 
we can again use a (weighted) blow-up similar to the polar coordinate map $\Phi$
if we recall that $\cosh^2(\varphi)-\sinh^2(\varphi)=1$. Indeed, we may just define the 
blow-up map by 
\benn
\Xi:\cB_x\ra \R^2,\qquad \Xi(\varphi,\rho)=(\rho \cosh\varphi,\rho\sinh \varphi)
\eenn
and apply it to our main example \eqref{eq:main}. As for the spherical polar blow-up, 
the map $\Xi$ induces a vector field, which we denote by $\hat{h}$, on $\cB_x$ by the 
requirement 
\benn
\Xi_*\left(\hat{h}\right)=f.
\eenn

\begin{lem}
\label{lem:hyp_polar}
The vector field $\hat{h}$ is given by
\be
\label{eq:hyp_polar}
\begin{array}{lcl}
\varphi'&=&\rho(3\sinh^2\varphi\cosh\varphi-2a\cosh^2\varphi \sinh\varphi),\\
\rho'&=&\rho^2(a\cosh\varphi-2\sinh\varphi-3\sinh^3\varphi-2a\cosh\varphi \sinh^2 \varphi).\\
\end{array}
\ee
\end{lem}

The proof of Lemma \ref{lem:hyp_polar} follows the same approach as Lemma \ref{lem:polar}.
As before, we may desingularize the vector field and consider 
\benn
\bar{h}:=\frac1\rho \hat{h}. 
\eenn
Then we look for steady states on $\H_x \times \{\rho=0\}$ and we have to solve
\benn
\sinh^2\varphi=\frac23 a\cosh\varphi \sinh\varphi
\eenn
since $\cosh\varphi\geq 1$. 

\begin{prop}
For the desingularized vector field $\bar{h}$, there is one steady state at $(\varphi,\rho)=(0,0)$ 
and a second one at $(\varphi,\rho)=\left(0,\textnormal{tanh}\left(\frac23a\right)\right)$. Both
points are hyperbolic saddles.
\end{prop}

The result is expected from the previous computations. Next, we observe that 
the geometry of the problem for the hyperbolic blow-up space $\cH_x$ is similar to the directional
blow-up in the $x$-direction; see Figure \ref{fig:fig3}.\medskip 

\begin{figure}[htbp]
	\centering
\psfrag{a}{(a)}
\psfrag{b}{(b)}
\psfrag{kappa}{$\nu_1$}
\psfrag{r1}{$r_1$}
\psfrag{y1}{$y_1$}
\psfrag{yt}{$\tilde{y}$}
\psfrag{xt}{$\tilde{x}$}
		\includegraphics[width=0.7\textwidth]{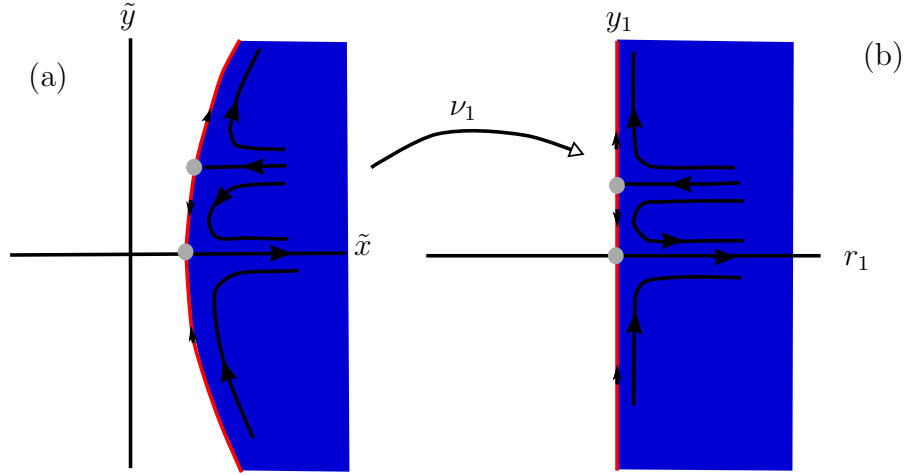}
		\caption{\label{fig:fig3}Sketch of the coordinate chart $\nu_1$ associated to the 
		$x$-directional blow-up. (a) Blown-up space $\cB_x=\H_x\times [0,\rho)$ with phase portrait 
		(black). (b) Directional
		coordinates $(r_1,y_1)\in\R^2$; the blue region corresponds to the blue region in (a) using the
		chart map $\nu_1$, respectively its inverse $\nu_1^{-1}$. Note that the curve 
		$\{\tilde{x}^2-\tilde{y}^2=1\}\times \{\rho=0\}$ from (a) is mapped to the vertical $y_1$-axis.}
\end{figure} 

Next, we check how to define the directional blow-ups based upon $\cB_x$. 
Let $(\tilde{x},\tilde{y},\tilde{\rho})$ be coordinates on $\cB_x$ with 
$\tilde{x}^2-\tilde{y}^2=1$ and $\tilde{\rho}\in[0,\rho_0]$. Define the blow-map
\benn
\Gamma(\tilde{x},\tilde{y},\tilde{\rho})=(\tilde{\rho}\tilde{x},\tilde{\rho}\tilde{y}).
\eenn
Let $\nu_i:\cB_x\ra \R^2$ be coordinate charts. As before, we want to construct 
the charts such that the local coordinate changes are given, as for the spherical 
case in \eqref{eq:psii}, by  
\be
\label{eq:psii1}
\gamma_1(r_1,y_1)=(r_1,r_1y_1)\qquad\text{and}\qquad \gamma_2(r_2,x_2)=(r_2x_2,r_2),   
\ee 
where $\gamma_i=\Gamma\circ\nu_i^{-1}$. In particular, the following diagram should commute 
\benn
\xymatrix{& & & \cB_x=\H_x\times[0,\rho_0] \ar[llld]^{\nu_2} \ar[ld]^{\nu_1} \ar[rd]^{\Gamma} & \\
 (r_2,x_2)\in\R^2 \ar[rr]_{\nu_{21}} & & (r_1,y_1)\in\R^2 \ar[ll]_{\nu_{12}} 
\ar[rr]^{\gamma_1} & & (x,y)\in\R^2,  
}
\eenn
where $\nu_{12}$, $\nu_{21}$ denote the transition maps. The conditions \eqref{eq:psii1} yield  
\be
\label{eq:def_charts2}
\begin{array}{l}
\nu_1(\tilde{x},\tilde{y},\tilde{\rho})=\gamma_1^{-1}\circ\Gamma(\tilde{x},\tilde{y},\tilde{\rho})=
\gamma_1^{-1}(\tilde{\rho}\tilde{x},\tilde{r}\tilde{y})=(\tilde{\rho}\tilde{x},\tilde{r}\tilde{y}/(\tilde{r}\tilde{x}))=
(\tilde{r}\tilde{x},\tilde{y}/\tilde{x}),\\
\nu_2(\tilde{x},\tilde{y},\tilde{\rho})=\gamma_2^{-1}\circ\Gamma(\tilde{x},\tilde{y},\tilde{\rho})=
\gamma_2^{-1}(\tilde{\rho}\tilde{x},\tilde{r}\tilde{y})=(\tilde{r}\tilde{x}/(\tilde{r}\tilde{y}),\tilde{\rho}\tilde{y})=
(\tilde{x}/\tilde{y},\tilde{r}\tilde{y}),\\
\end{array}
\ee
so the calculations are almost exactly the same as for the spherical case. However, there are some subtle 
differences when we consider the relation between the directional and hyperbolic polar blow-up maps. 
If we would like to change from the coordinates $(\varphi,\rho)$ to Euclidean coordinates $(r_1,y_1)$
we get the requirement
\benn
\Gamma(\varphi,\rho)=(\rho\cosh\varphi,\rho\sinh\varphi)=(x,y)=(r_1,r_1y_1)=\gamma_1(r_1,y_1).
\eenn 
Therefore, it follows that $r_1=\rho\cosh\theta$ which implies
\benn
r_1y_1=y_1\rho\cosh\varphi=\rho\sinh\varphi\quad\Rightarrow \quad y_1=\tanh\varphi.
\eenn
The coordinate change $\beta_1:\R^2\ra \R^2$ given by
\be
\beta_1(\varphi,\rho)=(\rho\cosh\varphi,\tanh\varphi)=(r_1,y_1)
\ee
is analytic and well-defined everywhere. Geometrically,
this is expected since we can easily map the domain 
\benn
\{\tilde{x}:\tilde{x}>0,\tilde{x}^2-\tilde{y}^2=1\}\times [0,\rho_0]
\eenn
diffeomorphically onto a rectangular strip of the form $\{(x,y):x\in[0,\rho_0]\}$;
see Figure \ref{fig:fig3}. For the 
second chart we get 
\benn
\Gamma(\varphi,\rho)=(\rho\cosh\varphi,\rho\sinh\varphi)=(x,y)=(r_2x_2,r_2)=\gamma_2(r_2,x_2).
\eenn 
Therefore, it follows that $r_2=\rho\sinh\theta$ which implies
\benn
r_2x_2=x_2\rho\sinh\varphi=\rho\cosh\varphi\quad\Rightarrow \quad x_2=\frac{1}{\tanh\varphi}.
\eenn
The coordinate change $\beta_2:\R^2\ra \R^2$ given by
\be
\beta_2(\varphi,\rho)=\left(\rho\sinh\varphi,\frac{1}{\tanh\varphi}\right)=(r_2,x_2)
\ee
is is not defined at $\varphi=0$ as $\tanh(0)=0$. Again, this is expected from the geometry
as shown in Figure \ref{fig:fig2}.\medskip

So we may conclude that the space $\cB_x$, which is built upon $\H_x$, basically yields
immediately a directional blow-up in the $x$-direction up to the analytic coordinate 
change $\beta_1$. Similarly, one may show that using $\cB_y$ corresponds, up to an analytic 
coordinate change, to a $y$-direction blow-up. As for the spherical case, we 
may define charts that also cover the negative half-planes.\medskip 

In summary, the example considered here demonstrates that the 
classical choice of a spherical blow-up in $\R^N$ with $\cS^{N-1}\times \cI$ for some 
interval $\cI\subseteq \R$ is certainly not the only option. In particular, if we already know a 
certain direction for $z\in\R^N$ where we do not need the directional blow-up, say $z_1$,
then hyperbolic space $\H_{z_1}$ is one good choice as it corresponds via an
analytic coordinate change to the respective directional blow-ups. Furthermore, the
analysis motivates that one should be aware that other manifolds, beyond spheres and hyperbolic
space, could also be used to construct a blow-up space. 

\bibliographystyle{plain}
\bibliography{../my_refs}

\end{document}